\documentclass[12pt,psfig,reqno]{amsart}

\usepackage{hyperref}

\usepackage{amssymb}
\usepackage{amsmath}
\usepackage{amsthm}
\usepackage{enumerate}
\usepackage{graphicx}
\usepackage{color}
\numberwithin{equation}{section}

\usepackage{pgfplots}
\usepackage{tikz}
\theoremstyle{plain}
\newtheorem{thm}{Theorem}[section]
\newtheorem*{thmA}{Theorem A}
\newtheorem*{thmB}{Theorem B}
\newtheorem{lem}[thm]{Lemma}
\newtheorem{prop}[thm]{Proposition}
\newtheorem{cor}[thm]{Corollary}

\newtheorem{de}[thm]{Definition}
\newtheorem{re}[thm]{Remark}

\begin{document}

\title[Frame bound, spectral gap and Plus space]
{Frame bound, spectral gap and Plus space}

\author{Zheng-Yi Lu}

\address{ School of Mathematics, Hunan University, Changsha Hunan 410082, China}

\email{zyluhnsd@163.com}

\thanks{The research is supported in part by the NNSF of China (Nos. 11831007 and 12071125), the Hunan Provincial NSF (No. 2025JJ60065), the Fundamental
	Research Funds for the Central Universities.}

\begin{abstract}
In this paper, we investigate the relationship between frame bounds and spectral gaps. By introducing the notion of \emph{essential minimum(maximal) spectral gap}, we provide a local characterization of Landau's theorem \cite{Lan67}. As an application, we resolve the spectrality additive measures of Lebesgue type,  conclusively answering an open question on the spectrality of Plus spaces originally raised by Lai, Liu, Prince \cite{LLP21} and further studied by Ai, Lu, Zhou  \cite{ALZ23} and Kolountzakis, Wu \cite{KW25}. 
\end{abstract}

\maketitle
\section{Introduction }
\subsection{Frame and spectral gap}
Let $\mu$ be a Borel probability measure with compact support in $\Bbb R^n$. We call $\mu$ a {\textit{frame-spectral measure}} on $\Bbb R$ if there exist $0<A\leq B<\infty$ and  $E(\Lambda):=\{e^{-2\pi i\lambda \cdot x}:\lambda\in\Lambda\}$ such that 
$$
A\int|f(x)|^2d\mu(x)\leq \sum_{\lambda\in\Lambda}\Big|\int f(x)e^{-2\pi i\lambda\cdot x}d\mu(x)\Big|^2\leq B\int|f(x)|^2d\mu(x)
$$
for any $f(x)\in L^2(\mu)$. 
In this case, the $E(\Lambda)$ is called a \textit{frame} of $\mu$ and $\Lambda$ is called a \textit{frame-spectrum} of $\mu$. 
Furthermore, we say that $E(\Lambda)$ is a \textit{tight frame} of $\mu$ when $A=B$. 

The theory of frames was first introduced by  Duffin and Schaeffer \cite{DS52} in 
1952 to study some fundamental problems in nonharmonic Fourier series. 
A significant advancement was made by Landau  \cite{Lan67}, who investigated the distribution of frame spectra. One of the key results in this area establishes that if \( \mathcal{L}_\Omega \) (the Lebesgue measure restricted to a set \(\Omega\)) admits a frame spectrum $\Lambda$, then the lower Beurling density of $\Lambda$ must satisfy 
\begin{align*}
	D^-(\Lambda)\geq \mathcal{L}(\Omega), 
\end{align*} where the lower Beurling density is defined as 
\[
D^-(\Lambda) = \liminf_{R \to \infty} \inf_{x \in \mathbb{R}^n} \frac{\#(\Lambda \cap B(x, R))}{|B(x, R)|}. 
\]
Here \( B(x, R) \) denotes the ball of radius \( R \) centered at \( x \), \( |B(x, R)| \) is its Lebesgue measure, and \( \#(\Lambda \cap B(x, R)) \) counts the points of \( \Lambda \) in \( B(x, R) \). 
This result reveals a fundamental density constraint: any frame spectrum $\Lambda$ must be sufficiently dense, independent of the frame bounds. For further reading on frame theory and spectral gaps, we refer to \cite{Chr03,Hei11,IP00, KL16, OS02}. 
However, an intuitive question arises: If the lower frame bound $A$ increases, should the gaps in the frame spectrum necessarily decrease? This motivates our central inquiry:
\\
	\textbf{(Q1): How do the frame bounds influence the spectral gaps of $\Lambda$?}

Let $\Lambda = \{\lambda_k\}_{k\in\mathbb{Z}}$ be a frame-spectrum for $\mu$ with $0 = \lambda_0 \in \Lambda$, where the elements are ordered such that $\lambda_n < \lambda_m$ for all $n < m$. 
Denote the spectral gaps as $g_k(\Lambda):=\lambda_k-\lambda_{k-1}$ for $k\in\Bbb Z$. 
To characterize the extremal behavior of these gaps, we define \textit{the essential minimal spectral gap} and \textit{the essential maximal spectral gap}, respectively, as:
\[
g_{\min}(\Lambda) := \inf\{c\geq0:g_k(\Lambda)<c\;\text{for\;infinitely\;many\;}k\}
\]
and 
\[
g_{\max}(\Lambda) := \sup\{c\geq0:g_k(\Lambda)>c\;\text{for\;infinitely\;many\;}k\}.
\]
Through a direct application of Landau's theorem to these gap characteristics, we establish the following fundamental result: 
\begin{thmA}[\cite{Lan67}]\label{thma}
Let $\Omega\subset\Bbb R$ satisfy $|\Omega|=1$. 
If $\mathcal{L}_{\Omega}$ is a frame-spectral measure with a frame-spectrum $\Lambda$, then $g_{\min}(\Lambda)\leq 1$. 
\end{thmA}

The estimates derived in the above theorem exhibit a remarkable independence from the frame bounds $A$ and $B$. We now bridge this gap by establishing an explicit connection between spectral gaps and frame bounds through Fourier-analytic techniques. For this analysis, we employ the Fourier transform characterization of the measure $\mu$, defined as:
\[
\widehat{\mu}(\xi) = \int e^{-2\pi i \xi \cdot x} \, d\mu(x),
\]
where the integral is taken over the support of $\mu$. 
\begin{thm}\label{thm1.1}
	Let $\mu$ be a Borel probability measure on $\Bbb R$. 
	Suppose that there exists $C>0$ such that \( |\widehat{\mu}(\xi)|\leq C|\xi|^{-1},\;\forall\xi\in\Bbb R\setminus\{0\}. \) 
If a countable set \(\Lambda\subset\Bbb R\) satisfies 
\begin{align}\label{(1.1')}
	\sum_{\lambda\in\Lambda}\Big|\int f(x)e^{-2\pi i\lambda x}d\mu(x)\Big|^2\geq A
	\int |f(x)|^2d\mu(x),\;\;\forall\;f(x)\in L^2(\mu),
\end{align}
then 
	\[
	g_{\min}(\Lambda) g_{\max}(\Lambda)\leq \frac{C^2\pi^2}{A}.
	\] 
\end{thm}
\begin{re}\label{thm1.2}
	Noting $g_{\min}(\Lambda)\leq g_{\max}(\Lambda)$, by the above theorem, then $	g_{\min}^2(\Lambda)\leq\frac{C^2\pi^2}{A}$. 
	If we modify the details of the proof, we can improve this estimate to $$	g_{\min}^2(\Lambda)\leq	\frac{C^2\pi^2}{3(A-1)}$$ when \( A>\frac32\).
\end{re}

Consider a family of sets 
 \(\{\Lambda_A\}\) that satisfy the frame condition \eqref{(1.1')}. 
The preceding remark establishes that the essential minimal spectral gap  \(g_{\min}(\Lambda_A)\) approaches $0$ as 
$A$ tends to infinity.
 A natural question then arises: What is the decay rate of \(g_{\min}(\Lambda_A)\) as \(A\) tends infinity? We investigate this question in the following theorem.

\begin{thm}\label{thm1.3}
	Let Borel probability measure $\mu=g(x)dx$ with $|\widehat{g}(\xi)|=\mathcal{O}(|\xi|^{-\alpha})$ for some $\alpha>\frac12$. 
	Suppose that there exists a family  $\{\Lambda_A:A>0\}$ such that 
	\begin{align*}
		\sum_{\lambda\in\Lambda_A}\Big|\int f(x)e^{-2\pi i\lambda x}d\mu(x)\Big|^2\geq A\int |f(x)|^2d\mu(x),\;\;\forall\;f(x)\in L^2(\mu).
	\end{align*}
Then 
	\[	 \varlimsup_{A\rightarrow\infty}A g_{\min}(\Lambda_A)\leq	\int |g(x)|^2dx. \] 
\end{thm}
By specializing to the classical Lebesgue measure restricted to the unit interval $[0,1]$, we immediately derive the following corollary.  
\begin{cor}\label{thm1.4}
	Let $\mu$ be the Lebesgue measure restricted to $[0,1]$. 
		Suppose that there exists a family  $\{\Lambda_A:A>0\}$ such that 
	\begin{align*}
		\sum_{\lambda\in\Lambda_A}\Big|\int_{[0,1]} f(x)e^{-2\pi i\lambda x}dx\Big|^2\geq A\int_{[0,1]} |f(x)|^2dx,\;\;\forall\;f(x)\in L^2([0,1]).
	\end{align*}
	Then 
	\[	 \varlimsup_{A\rightarrow\infty}A g_{\min}(\Lambda_A)\leq	1. \] 
\end{cor}
This estimate is in fact sharp, as demonstrated by the existence of a family $\{\Lambda_A\}_{A>0}$ achieving the limiting behavior:
\[
\lim_{A\to\infty} A\cdot g_{\min}(\Lambda_A) = 1
\]
(See Proposition~\ref{thm2.1} for the detailed construction).

For the classical Lebesgue measure, Landau's fundamental result \cite{Lan67} established that any frame spectrum must satisfy certain density requirements. This leads to two natural questions regarding spectral sparsity:
\begin{enumerate}
	\item How does the frame bounds $A,B$ influence the maximal gap size $g_{\max}(\Lambda)$?
	\item What is the precise quantitative relationship between these parameters?
\end{enumerate}
We now establish this connection through the following theorem. Together with Theorem \ref{thm1.1}, these results will enable us to completely resolve the spectrality problem for the additive measures of Lebesgue type in the subsequent subsection.

\begin{thm}\label{thm1.5}
Suppose that there exist $0<A\leq B<\infty$ such that 
	\[
 A\int_{[0,1]} |f(x)|^2dx\leq\sum_{\lambda\in\Lambda}\Big|\int_{[0,1]} f(x)e^{-2\pi i\lambda x}dx\Big|\leq B\int_{[0,1]} |f(x)|^2dx,\;\;\forall\;f(x)\in L^2([0.1]). 
	\]
Then \[
\frac{4}{\pi^2 B}\leq g_{\max}(\Lambda)\leq \sup\{g_k(\Lambda):k\in\Bbb Z\}\leq\frac{\pi^2B}{A}+2.
\]
\end{thm}
\subsection{Spectral measure and Plus space}
We say that $\mu$ is a {\textit{spectral measure}} if $L^2(\mu)$ admits a Fourier orthogonal basis $E(\Lambda)$, in which case $\Lambda$ is called a spectrum of $\mu$. 
The study of spectral measures originated with Fuglede?s work \cite{Fug74}, which introduced the famous spectral set conjecture: a set $\Omega \subset \mathbb{R}^n$ is spectral if and only if it is a translational tile. While this conjecture has been disproven for dimensions $n \geq 3$ by Tao and others \cite{KM06, Ma05, Tao}, it remains open for $n = 1$ and $2$. To this day, it continues to be an active research topic \cite{FFLS19,IKT03, La01, LM22}. Recently, the spectrality of singular measures has also attracted increasing attention.
A major breakthrough came in 1998, when Jorgensen and Pedersen \cite{JP98} constructed the first singular, non-atomic spectral measure-the $1/4$-Cantor measure. Their discovery sparked significant interest in the spectral theory of fractal measures. For further details and recent progress, we refer to \cite{Dai16, Dai12, DHL14}, among others.

In 2018, Lev \cite{Lev18} studied the addition of two
measures supported respectively on two orthogonal subspaces embedded in the ambient
space $\mathbb R^n$  and showed that these measures admit Fourier frames. Recently, Lai, Liu and Prince \cite{LLP21} studied the spectrality of additive measures. 
\begin{de}
	Let $\mu$ and $\nu$ be two continuous Borel probability measures on $\mathbb{R}$. We embed them into the $x$ and $y$ axes in $\mathbb{R}^{2}$ respectively. The \emph{additive space over $\mu$ and $\nu$} is the space $L^{2}(\rho)$, where $\rho$ is the measure
	
	\[
	\rho=\frac{1}{2}(\mu\times\delta_{0}+\delta_{0}\times\nu),
	\] 
	and $\delta_{0}$ is the Dirac measure at $0$. We will refer to the (compact) support of $\mu$ and $\nu$ as the \emph{component spaces} of the measure $\rho$. If $\mu=\nu$, we say that $\rho$ is \emph{symmetric}. 
	If $\mu,\nu$ are Lebesgue measures supported on intervals of length one, we call $\rho$ the \emph{additive measure
		 of Lebesgue type}. 
\end{de}
\noindent
\\
For convenience, we denote the additive measure of Lebesgue type as
\begin{align}\label{(1.2)}
	\rho_{t_1,t_2}=\frac{1}{2}(\mathcal{L}_{[t_1,t_1+1]}\times\delta_0+\delta_0\times\mathcal{L}_{[t_2,t_2+1]}).
\end{align}
In the symmetric case where $t_1 = t_2 = t$, we simply write $\rho_t$. 
Lai, Liu and Prince \cite{LLP21} investigated the spectral properties of these measures and posed the following open question:
\\
 \textbf{(Q2): When is \(\rho_{t_1,t_2}\) a spectral measure? In particular, does the  Plus space (i.e., $L^2(\rho_{-\frac{1}{2}})$) admit an exponential orthogonal basis?}

This question has attracted considerable attention from researchers. Ai, the author, and Zhou established a necessary and sufficient condition for $\rho_t$ to be a spectral measure when $t \in \mathbb{Q}\setminus{-\frac{1}{2}}$. Recently, Kolountzakis and Wu \cite{KW25} extended these results by proving that $\rho_t$ fails to be a spectral measure for all irrational values of $t$. We summarize these findings as follows:
\begin{thmB}[\cite{ALZ23,KW25}]\label{thmb}
	Let $\rho_{t}$ be the symmetric additive measure of Lebesgue type, where the component measure is the Lebesgue measure supported on $[t,t+1]$. 
	Suppose that $t\neq-\frac12$. 
	Then $\rho_{t}$ is a spectral measure if and only if $t\in\frac12\Bbb Z\setminus\{-\frac12\}$. 
\end{thmB}

From these results, we observe that the spectral property of \(\rho_{t}\) is now completely characterized except for the case \(t = -\frac{1}{2}\). 
To address \textbf{(Q2)}, we employ Theorem  \ref{thm1.1} to settle the spectrality of Plus space(Proposition \ref{thm4.3}).  Moreover, we extend our investigation to more general additive measures of Lebesgue type, establishing the following result:
\begin{thm}\label{thm1.7}
Let $\rho_{t_1,t_2}$ be given as \eqref{(1.2)}. 
Then $\rho_{t_1,t_2}$ is a spectral measure if and only if $t_1+t_2\in\Bbb Z\setminus\{-1\}$ or $t_1-t_2\in\Bbb Z\setminus\{0\}
$. 
\end{thm}

The paper is organized as follows. In Section 1, we present the introduction and state our main results. Section 2 focuses on the proofs of Theorems \ref{thm1.1} and \ref{thm1.3}, which are established through an analysis of the Hurwitz zeta function. The proof of Theorem \ref{thm1.5} is provided in Section 3. Finally, in Section 4, we completely resolve \textbf{(Q2)} by applying Theorems \ref{thm1.1} and \ref{thm1.5}, and subsequently prove Theorem \ref{thm1.7}.

\section{Proof of Theorems \ref{thm1.1} and \ref{thm1.3}}
In this section, we use some estimates of Hurwitz zeta function to prove Theorem \ref{thm1.1}. 
Furthermore, we complete the proof of Theorem \ref{thm1.3}. 

Recall that the Fourier transform of a finite Borel measure is defined to be 
$$
\widehat{\mu}(\xi)=\int e^{-2\pi i\xi\cdot x}d\mu(x).
$$
\begin{proof}[\textbf{Proof of Theorem \ref{thm1.1}}]
Let $f(x)=e^{2\pi i \xi x}$ in \eqref{(1.1')}. 
It follows that 
\begin{align}\label{(2.1)}
	A\leq \sum_{\lambda\in\Lambda}|\widehat{\mu}(\lambda-\xi)|^2\leq 
	C^2\sum_{\lambda\in\Lambda}\frac{1}{|\lambda-\xi|^2}
	=	C^2\sum_{k\in\Bbb Z}\frac{1}{|\lambda_k-\xi|^2} 
\end{align}
for any $\xi\in\Bbb R\setminus\Lambda$, where $\Lambda=\{\lambda_k:k\in\Bbb Z\}$ and $\cdots<\lambda_{-1}<\lambda_0<\lambda_1<\cdots$. 
Recall that $g_k(\Lambda)=\lambda_k-\lambda_{k-1}$, 
\[
g_{\min}(\Lambda) = \inf\{c\geq0:g_k(\Lambda)<c\;\text{for\;infinitely\;many\;}k\}
\]
and 
\[
g_{\max}(\Lambda) = \sup\{c\geq0:g_k(\Lambda)>c\;\text{for\;infinitely\;many\;}k\}.
\]
Fixed small $\epsilon>0$, there exist infinitely many $k\in\{k_j\}_{j=1}^\infty$ such that $g_k(\Lambda)>d_1:=g_{\max}(\Lambda)(1-\epsilon)$. 
Put $\xi_j=\frac12(\lambda_{k_j-1}+\lambda_{k_j})$ for each $j\geq1$. 
Denote that $d_2=g_{\min}(\Lambda)(1-\epsilon)$, which implies that  $\#\{k:g_k(\Lambda)<d_2\}<\infty$. 
So we can choose some large $K>0$ such that $d_2\leq g_k(\Lambda)$ for all $|k|>K$. 
Then the \eqref{(2.1)} can be written as 
\begin{align}\label{(2.2)}
	A\leq C^2\sum_{k\in\Bbb Z}\frac{1}{|\xi_{j}-\lambda_k|^2}
	&\leq C^2\sum_{\ell=0}^\infty\Big(\frac{1}{(\xi_j-\lambda_{k_j+\ell})^2}+\frac{1}{(\xi_j-\lambda_{k_j-\ell-1})^2}\Big)
	\\
	&\notag\leq C^2\Big(\sum_{\ell=0}^\infty\frac{2}{(\frac12d_1+\ell d_2)^2}+\sum_{|k|\leq K}\frac{1}{|\xi_{j}-\lambda_k|^2}\Big).
\end{align}
Letting $j\rightarrow\infty$, one has $$\sum_{|k|\leq K}\frac{1}{|\xi_{j}-\lambda_k|^2}\rightarrow0,$$
which yields that 
\begin{align*}
	\frac{Ad_2^2}{2C^2}\leq  \sum_{\ell=0}^\infty\frac{1}{(x+\ell)^2}:=\zeta(2,x),
\end{align*}
where $x:=\frac{d_1}{2d_2}\geq \frac12$ and Hurwitz zeta function $\zeta(s,x)=\sum_{n=0}^\infty\frac{1}{(x+n)^s}$. 
Hence 
\begin{align*}
d_1d_2=2xd_2^2\leq \frac{4C^2}{A}x\zeta(2,x):=\frac{4C^2}{A}F(x).
\end{align*}
We claim that $F(x)\leq F(\frac12)=\frac14\pi^2$ for $x\geq\frac12$. 
This yields that $d_1d_2\leq \frac{C^2\pi^2}{A}$. 
Therefore 
$$
g_{max}(\Lambda)g_{min}(\Lambda)(1-\epsilon)^2\leq \frac{C^2\pi^2}{A}.
$$
For the arbitrariness of $\epsilon$, one has 
$g_{max}(\Lambda)g_{min}(\Lambda)\leq\frac{C^2\pi^2}{A}$.

In the following, we prove the claim. 
For $x\geq 0.65$, observe that 
\begin{align*}
	F(x)\leq x(\frac1x+\frac{1}{(x+1)^2})+x\int_{x+1}^\infty\frac{1}{t^2}dt
	=1+\frac{1}{x}-\frac{1}{(x+1)^2}:=f(x).
\end{align*}
It is easy to check that $f'(x)<0$  in $[0.65,\infty)$, i.e., $f(x)$ is decreasing, so 
\begin{align}\label{(6)}
	F(x)\leq f(x)\leq f(0.65)<\frac{\pi^2}{4}=F(\frac12). 
\end{align}

Now, we consider $x\in[0.5,0.65]$. 
We know that $\zeta(x,2)$ is analysis when $x\in[0.5,0.65]$. Therefore, the function $F(x)=x\zeta(x,2)$ admits a Taylor series expansion about $\frac12$:
\begin{equation*}
	F(x)-F(\frac12) = \sum_{n=1}^{\infty} \frac{F^{(n)}(\frac12)}{n!}(\Delta x)^n
\end{equation*}
where $F^{(n)}(\frac12)$ denotes the $n$-th derivative of $F$ evaluated at $\frac12$ and $\Delta x=x-\frac12$. 
Note that 
\begin{align*}
	F^{(n)}(\frac12)&=(-1)^n n! \left( \frac12(n+1) \sum_{\ell=0}^\infty \frac{1}{(\frac12+\ell)^{n+2}} - n\sum_{\ell=0}^\infty \frac{1}{(\frac12+\ell)^{n+1}} \right)
	\\
	&=(-1)^n n! \left( \frac12(n+1) \zeta(n+2,\frac12) - n\zeta(n+1,\frac12) \right).
\end{align*}
Then  
\begin{align*}
	&F(x)-F(\frac12) = \sum_{k=1}^{\infty} \left(-k(1+2\Delta x)\eta_{2k+1}+(2k-1)\eta_{2k}+(k+\frac12)\eta_{2k+2}\Delta x\right)\Delta x^{2k-1}
\end{align*}
with $\eta_s=\zeta(s,\frac12)$. 
It follows from  $\zeta(s,\frac12)=(2^s-1)\zeta(s)$ and $\zeta(s)\in(1,\frac{s}{s-1})$ for $s\geq2$  that
\begin{align*}
	&\sum_{k=1}^{\infty} \left(-k(1+2\Delta x)\eta_{2k+1}+(2k-1)\eta_{2k}+(k+\frac12)\eta_{2k+2}\Delta x\right)\Delta x^{2k-1}
	\\
	\leq& \sum_{k=1}^{\infty} \left(k(2\Delta x-1)+4\Delta x2^{2k}(-k+(k+1)\Delta x)-(k+1)\Delta x\right)\Delta x^{2k-1}
	\\
	\leq& 0, 
\end{align*}
 where $\zeta(s)$ is the well-known Riemann function. 
Hence $
	F(x)-F(\frac12)\leq 0
$ for $x\in[0.5,0.65]$, which together with \eqref{(6)} implies that the claim holds. 
This completes the proof. 
\end{proof}
\begin{proof}[\textbf{Proof of Remark 1.2}]
	In the above proof, if we choose $\xi_j=\lambda_{k_j}$, then \eqref{(2.2)}
	can be \begin{align*}
		A\leq C^2\sum_{k\neq k_j}\frac{1}{|\xi_{j}-\lambda_k|^2}+1
		&\leq 1+C^2\sum_{\ell=1}^\infty\Big(\frac{1}{(\xi_j-\lambda_{k_j+\ell})^2}+\frac{1}{(\xi_j-\lambda_{k_j-\ell})^2}\Big)
		\\
		&\notag\leq 1+C^2\Big(\sum_{\ell=1}^\infty\frac{2}{\ell^2 d_2^2}+\sum_{|k|\leq K}\frac{1}{|\xi_{j}-\lambda_k|^2}\Big).
	\end{align*}
	So, letting $j\rightarrow\infty$, 
	$$
	A-1\leq \frac{C^2\pi^2}{3d_2^2},
	$$
	that is $g_{\min}^2(\Lambda)(1-\epsilon)^2\leq \frac{C^2\pi^2}{3(A-1)}$. 
	The arbitrariness of $\epsilon$ gives $g_{\min}^2(\Lambda)\leq \frac{C^2\pi^2}{3(A-1)}$, which completes the proof. 	
\end{proof}
The preceding arguments demonstrate that some  decay condition on $|\widehat{\mu}(\xi)|$ suffices to guarantee $g_{\min}(\Lambda) \to 0$ as $A \to \infty$. We now quantify this decay relationship and establish Theorem \ref{thm1.3}.
\begin{proof}[\textbf{Proof of Theorem \ref{thm1.3}}]
	Let $\Lambda_A$ satisfy 
	\begin{align*}
		\sum_{\lambda\in\Lambda_A}\Big|\int f(x)e^{-2\pi i\lambda x}d\mu(x)\Big|\geq A\int|f(x)|^2d\mu(x),\;\;\forall\;f(x)\in L^2(\mu)\}, 
	\end{align*}
	and $d_\epsilon:=g_{\min}(\Lambda_A)(1-\epsilon)
	$ for a small $0<\epsilon<\frac12$. 
	Choose some $f(x)=e^{2\pi i\xi x}$, then the above  inequality becomes 
	\begin{align*}
		A\leq \sum_{\lambda\in\Lambda_A}|\widehat{\mu}(\xi-\lambda)|^2.
	\end{align*}
	We decompose $\Lambda_A$ into $\Lambda_A' \cup \widetilde{\Lambda_A}$, where the elements in $\Lambda_A'$ are separated by at least $d_\epsilon$ and $\widetilde{\Lambda_A}$ contains finitely many elements.
It follows from $|\widehat{\mu}(\xi)|=\mathcal{O}(|\xi|^{-\alpha})$ with $\alpha>\frac12$ that there exists $C>0$ such that 
	\begin{align}\label{(2.4)}
		A\leq \sum_{\lambda\in\Lambda_A}|\widehat{\mu}(\xi-\lambda)|^2\leq 
		\sum_{\lambda\in\widetilde{\Lambda_A}}\frac{C^2}{|\xi-\lambda|^{2\alpha}}+\sum_{\lambda\in\Lambda_A'}|\widehat{\mu}(\xi-\lambda)|^2
	\end{align}
	for $\xi\not\in\widetilde{\Lambda_A}$. 
	Choose large $\xi$ satisfying $\sum_{\lambda\in\widetilde{\Lambda_A}}\frac{C^2}{|\xi-\lambda|^{2\alpha} }\leq 1$, as the term tends zero when $\xi\rightarrow\infty$. 

We claim that $g_{\min}(\Lambda_A)\rightarrow0$, as $A\rightarrow\infty$, and we prove this by a contradiction. 
Suppose that there exists $\varepsilon>0$ such that $
\varlimsup_{A\rightarrow\infty}g_{\min}(\Lambda_A)\geq 4\varepsilon,
$ i.e., we can find a sequence $\{A_n\}$ satisfying 
$$
g_{\min}(\Lambda_{A_n})\geq 2\varepsilon
$$
for $n\geq1$, which implies that  $d_\epsilon\geq\frac12g_{\min}(\Lambda_{A_n})\geq \varepsilon$. 
According to \eqref{(2.4)} and $|\widehat{\mu}(\xi-\lambda)|\leq \frac{C}{|\xi-\lambda|^\alpha}$, one can find large $\xi_n$ such that 
\begin{align*}
			A_n\leq \sum_{\lambda\in\Lambda_{A_n}}|\widehat{\mu}(\xi_n-\lambda)|^2\leq 
2+\sum_{k=1}^\infty \frac{2C^2}{k^{2\alpha}d_\epsilon^{2\alpha}}
\leq 2+\frac{2C^2}{\varepsilon^{2\alpha}}\sum_{k=1}^\infty \frac{1}{k^{2\alpha}}.
\end{align*}
This is impossible when $n\rightarrow\infty$. 

The real axis is partitioned into a grid of unit length \( d_\epsilon \), that is,
\[
\mathbb{R} =\bigcup_{n \in \mathbb{Z}}I_n:= \bigcup_{n \in \mathbb{Z}} \big[ n d_\epsilon, (n+1) d_\epsilon \big].
\]
It is easy to see that \( I_n \) has at most one intersection with \( \xi-\Lambda_A' \) for a fixed $\xi\in\Bbb R$ and any $n\in\Bbb Z$, which implies that $\#((\xi-\Lambda_A')\cap I_n)\leq 1$ for all $n\in\Bbb Z$. 
Hence, choose some large $\xi$, from \eqref{(2.4)}, 
	\begin{align*}
	A\leq \sum_{\lambda\in\Lambda_A}|\widehat{\mu}(\xi-\lambda)|^2&\leq 
	\sum_{n\in\Bbb Z}\sum_{\lambda\in (\xi-I_n)\cap\Lambda_A'}|\widehat{\mu}(\xi-\lambda)|^2+1\leq 	\sum_{n\in\Bbb Z}\max_{x\in I_n}|\widehat{\mu}(x)|^2+1.
\end{align*}
Multiplying both sides by $d_\epsilon$ yields
\begin{align*}
	Ad_\epsilon\leq d_\epsilon+d_\epsilon\sum_{n\in \Bbb Z}\max_{\xi\in I_n}|\widehat{\mu}(\xi)|^2=d_\epsilon+\sum_{n\in \Bbb Z}|I_n||\widehat{\mu}(x_n)|^2
\end{align*}
for some $x_n\in I_n$. 
Letting $A\rightarrow\infty$, by the claim, one has 
\begin{align*}
\varlimsup_{A\rightarrow\infty}(1-\epsilon)Ag_{\min}(\Lambda_A)&\leq \varlimsup_{A\rightarrow\infty}((1-\epsilon)g_{\min}(\Lambda_A)+\sum_{n\in \Bbb Z}|I_n||\widehat{\mu}(x_n)|^2)
\\
&=\varlimsup_{|I_n|\rightarrow0}\sum_{n\in \Bbb Z}|I_n||\widehat{\mu}(x_n)|^2
\\
&=\int_{-\infty}^{\infty} |\widehat{\mu}(x)|^2dx.
\end{align*}
According to Plancherel's Theorem, one has 
$$
\int_{-\infty}^{\infty}|\widehat{\mu}(x)|^2dx
=\int_{-\infty}^{\infty}|\widehat{g}(x)|^2dx
=\int_{-\infty}^{\infty}|g(x)|^2dx, 
$$
which together with the arbitrariness of $\epsilon$ yields that 
$$
\varlimsup_{A\rightarrow\infty}Ag_{\min}(\Lambda_A)\leq \int_{-\infty}^{\infty}|g(x)|^2dx,
$$
as desired.  
\end{proof}

We take $g(x)$ as the characteristic function restricted to the interval $[0,1]$ and obtain the following conclusion.
\begin{prop}\label{thm2.1}
	There exists $\{\Lambda_A:A>0\}$ such that 	\begin{align*}
		\sum_{\lambda\in\Lambda_A}\Big|\int_{[0,1]} f(x)e^{-2\pi i\lambda x}dx\Big|^2\geq A\int_{[0,1]}|f(x)|^2dx,\;\;\forall\;f(x)\in L^2([0.1])\}, 
	\end{align*}
	and 
	\begin{align*}
		\lim_{A\rightarrow\infty}Ag_{\min}(\Lambda_A)=1.
	\end{align*}
\end{prop}
\begin{proof}
	For $A>0$, there exists a unique integer $n$ such that $n-1\leq A<n$, write $\Lambda_A:=\cup_{j=0}^{n-1}\{\frac{j}{n}+\Bbb Z:j=0,1,\cdots,n-1\}$. 
	It follows from 
\begin{align*}
	\sum_{\lambda\in\frac{j}{n}+\Bbb Z}\Big|\int_{[0,1]} f(x)e^{-2\pi i\lambda x}dx\Big|=\int_{[0,1]}|f(x)|^2dx,\;\;\;\forall j=0,1,\cdots,n-1
\end{align*} 
that 
\begin{align*}
	\sum_{\lambda\in\Lambda_A}\Big|\int_{[0,1]} f(x)e^{-2\pi i\lambda x}dx\Big|=n\int_{[0,1]}|f(x)|^2dx\geq A\int_{[0,1]}|f(x)|^2dx.
\end{align*}
Hence the $\Lambda_A$ satisfies the first conclusion of the proposition, and 
$g_{\min}(\Lambda_A)=\frac1n$. 
Then 
	\begin{align*}
\frac{n-1}{n}\leq Ag_{\min}(\Lambda_A)\leq 1, 
\end{align*}
which yields that $\lim_{A\rightarrow\infty}Ag_{\min}(\Lambda_A)=1$, and the proposition follows. 
\end{proof}
\section{The estimation of the essential maximal spectral gap}
For a frame spectrum \(\Lambda\), while estimates of \( g_{\min}(\Lambda) \) reveal that the spectrum can have arbitrarily small spacing, they do not provide information about the density of the spectrum (i.e., the number of spectral points per unit interval). In this section, we quantify the spectral point density and investigate its relationship with the maximal essential spectral gap and frame bounds.

When a set $E(\Lambda)$ satisfies only the upper frame inequality for the measure $\mu$, we refer to $E(\Lambda)$ as a \textit{Bessel sequence} for $\mu$. For notational convenience, we write $\lfloor x \rfloor$ for the greatest integer less than or equal to $x$.
\begin{lem}\label{thm3.1}
Let $E(\Lambda)$ be a Bessel sequence of $\mathcal{L}_{[0,1]}$ with  bounded $B$. 
For any $\xi\in\Bbb R$, then $\#(\Lambda\cap[\xi-\frac12,\xi+\frac12])\leq \lfloor \frac{\pi^2}{4}B \rfloor$. 
\end{lem}
\begin{proof}
	Let $x\in\Bbb R$ and $f(x)=e^{2\pi i\xi x}$.  
	It follows that 
	\begin{align*}
		\sum_{\lambda\in\Lambda}\Big|\int_{[0,1]}f(x)e^{-2\pi i\lambda x}dx\Big|^2=\sum_{\lambda\in\Lambda}|\widehat{\mathcal{L}_{[0,1]}}(-\xi+\lambda)|^2\leq B.
	\end{align*}
	Note that \begin{align*}
		\sum_{\lambda\in\Lambda}|\widehat{\mathcal{L}_{[0,1]}}(-\xi+\lambda)|^2=		\sum_{\lambda\in\Lambda}\Big|\frac{\sin\pi(\xi-\lambda)}{\pi(\xi-\lambda)}\Big|^2\geq \frac{4}{\pi^2}\#(\Lambda\cap[\xi-\frac12,\xi+\frac12])
	\end{align*}
	Hence 
	\begin{align*}
	\#(\Lambda\cap[\xi-\frac12,\xi+\frac12])\leq \frac{\pi^2}{4}B. 	
	\end{align*}
	So the lemma follows from that $\#(\cdot)$ is an integer. 
\end{proof}
\begin{proof}[\textbf{Proof of Theorem \ref{thm1.5}}]
We first prove the estimate for the upper bound. 
		Choose some $\lambda_k\in\Lambda$ satisfying $$
		\lambda_{k+1}-\lambda_k\geq \sup\{g_k(\Lambda)\}-\epsilon
		$$
		for a small $\epsilon>0$. 
		Let $f(x)=e^{-2\pi i\xi}$ with $\xi=\frac{1}{2}(\lambda_k+\lambda_{k+1})$. 
It follows from Lemma \ref{thm3.1} that 
		\begin{align*}
			A\int_{[0,1]}|f(x)|^2dx\leq \sum_{\lambda\in\Lambda}|\widehat{\mathcal{L}_{[0,1]}}(\xi+\lambda)|^2&\leq \sum_{j=0}^\infty(\sum_{\lambda\in I_j}\frac{1}{(\xi-\lambda)^2}+\sum_{\lambda\in \widetilde{I}_j}\frac{1}{(\xi-\lambda)^2})
			\\
			&\leq \sum_{j=0}^\infty\frac{\#(I_j\cup\widetilde{I_j})}{(0.5g_{k}(\Lambda)+j)^2}
					\\
			&\leq 2\lfloor \frac{\pi^2}{4}B \rfloor \zeta(2,0.5g_k(\Lambda)),
		\end{align*}
		where $I_j:=[\lambda_{k+1}+j,\lambda_{k+1}+j+1]\cap\Lambda$ and $\widetilde{I}_j:=[\lambda_{k}-j-1,\lambda_{k}-j]\cap\Lambda$ for $j\geq0$. 
	Thus 
	\begin{align*}
		g_k(\Lambda)\leq 2\zeta^{-1}(2,2A^{-1}\lfloor \frac{\pi^2}{4}B \rfloor),
	\end{align*}	
	where the notation $\zeta^{-1}(2,x)$ denotes the inverse function of $\zeta(2,x)$. 
	A simple calculation gives 
	\begin{align*}
		g_k(\Lambda)\leq \frac{\pi^2B}{2A}(1+\sqrt{1+\frac{8A}{\pi^2B}}).
	\end{align*}
	Letting $\epsilon\rightarrow0$, we obtain that 
	$$\sup\{g_k(\Lambda)\}\leq \frac{\pi^2B}{2A}(1+\sqrt{1+\frac{8A}{\pi^2B}})< \frac{\pi^2B}{A}+2.$$
	
	In the following, we prove the theorem by showing $g_{\max}(\Lambda)\geq \frac{4}{\pi^2B}$. 
If $[\lambda,\lambda+1]\cap\Lambda=\{\lambda\}$ for infinitely many $\lambda\in\Lambda$, then $$g_{\max}(\Lambda)\geq 1>\frac{4}{\pi^2B}.
$$ 
If not, we assume $[\lambda,\lambda+1]\cap\Lambda=\{\lambda,\lambda_1,\cdots,\lambda_s\}$ for some $s\geq1$ and $\lambda:=\lambda_0<\lambda_1<\cdots<\lambda_s$.We define $\lambda_{s+1}$ to be the immediate successor of $\lambda_s$ in $\Lambda$.
Hence, using Lemma \ref{thm3.1} again,  
\begin{align*}
	\max\{\lambda_j-\lambda_{j-1}:1\leq j\leq s+1\}\geq \frac{\lambda_{s+1}-\lambda}{s+1}\geq \frac{1}{\#(\Lambda\cap[\lambda,\lambda+1])}\geq \frac{4}{\pi^2B}, 
\end{align*}
as desired. 
\end{proof}

A direct application of Theorem \ref{thm1.5} yields the following corollary.

\begin{cor}\label{thm3.2}
	Let $\Lambda$ be a tight frame-spectrum of $\mathcal{L}_{[0,1]}$. 
	Then \[
 g_{\max}(\Lambda)\leq \sup\{g_k(\Lambda):k\in\Bbb Z\}\leq\pi^2+2.
	\]
\end{cor}

\section{Spectrality of the additive measure of Lebesgue type}
In this section, we establish the spectral properties of the additive Lebesgue-type measure, thereby completing the proof of Theorem \ref{thm1.7}. Notably, we demonstrate that the Plus space does not admit an exponential orthogonal basis (Proposition \ref{thm4.3}).

Recall the additive measure of Lebesgue type defined as 
\begin{align*}
	\rho_{t_1,t_2}=\frac{1}{2}(\mathcal{L}_{[t_1,t_1+1]}\times\delta_0+\delta_0\times\mathcal{L}_{[t_2,t_2+1]}).
\end{align*}
To study the spectrality of $\rho_{t_1,t_2}$, we analyze the orthogonal structure of $L^2(\rho_{t_1,t_2})$.
Suppose $\Lambda\subseteq \Bbb R^2$ is a spectrum of $\rho_{t_1,t_2}$. Then, for any $\lambda',\lambda''\in \Lambda$, 
$$
\lambda'-\lambda''\in\mathcal{Z}(\widehat{\rho_{t_1,t_2}})\cup\{\textbf{0}\}, 
$$
where $\mathcal{Z}(\widehat{\rho_{t_1,t_2}})$ denotes the zero set of the Fourier transform  $\widehat{\rho_{t_1,t_2}}$.
From \cite{LLP21}, we know that $(\lambda_1,\lambda_2)\in\mathcal{Z}(\widehat{\rho_{t_1,t_2}})$ if and only if   
\begin{align*}
	e^{\pi(\lambda_1(2t_1+1)-\lambda_2(2t_2+1))}\frac{\sin\pi\lambda_1}{\pi\lambda_1}+\frac{\sin\pi\lambda_2}{\pi\lambda_2}=0,
\end{align*}
which implies that  
\begin{align}\label{(4.1)}
(\lambda_1,\lambda_2)\in(\Bbb Z\setminus\{0\})^2\;\text{or}\;\begin{cases}
		T(\lambda_1,\lambda_2)\in\Bbb Z;
		\\
		(-1)^{	T(\lambda_1,\lambda_2)}\frac{\sin\pi\lambda_1}{\pi\lambda_1}+\frac{\sin\pi\lambda_2}{\pi\lambda_2}=0,
	\end{cases}
\end{align}
where $T(\lambda_1,\lambda_2):=\lambda_1(2t_1+1)-\lambda_2(2t_2+1)$. 
\begin{lem}\label{thm4.1}
Let $\rho_{t_1,t_2}$ be given as \eqref{(1.2)} and  $(\lambda_1,\lambda_2)\in\mathcal{Z}(\widehat{\rho_{t_1,t_2}})$. 
If $|\lambda_1|\leq0.8$, then $\lambda_1=\lambda_2$ or $\lambda_1+\lambda_2=0$. 
\end{lem}
\begin{proof}
A simple estimation allows us to obtain that \begin{align}\label{(4.2)}
\alpha:=\frac{\sin\pi\lambda_1}{\pi\lambda_1}\geq \frac{\sin0.8\pi}{0.8\pi}\geq 0.23
\end{align}
 when $|\lambda_1|\leq 0.8$. 
From \eqref{(4.1)} and the condition $|\lambda_1|\leq0.8$, one has 
$T(\lambda_1,\lambda_2)\in\Bbb Z$ and $$
(-1)^{	T(\lambda_1,\lambda_2)}\frac{\sin\pi\lambda_1}{\pi\lambda_1}+\frac{\sin\pi\lambda_2}{\pi\lambda_2}=0.
$$
Write $f(x):=	\frac{\sin x}{ x}$ for $x\neq0$. 
Then assumption $(\lambda_1,\lambda_2)\in\mathcal{Z}(\widehat{\rho_{t_1,t_2}})$ allow us to know that $\lambda_2$ is the solution of $f(\pi x)=\pm\alpha$, i.e., $\sin\pi x=\pm\alpha \pi x$. 
Therefore, $\pi\lambda_2$ is the $x$-coordinate of the intersection point between the sine function $\sin x$ and the straight line $\alpha x$ (or $-\alpha x$). As can be seen from the figure, this straight line must lie between the two tangent lines (shown as the red dashed lines in \textbf{Figue 1}) when $\lambda_1\pm\lambda_2\neq0$. Thus, the absolute value of the line's slope should satisfy
\begin{equation*}
	|\alpha| \leq |f(x_0)|
\end{equation*}
where $x_0$ is the minimal positive solution to the equation
\begin{equation*}
	\tan x_0 = x_0.
\end{equation*}
This is impossible, as $|f(x_0)|\leq 0.22<\alpha$ and \eqref{(4.2)}. 
It yields that $\lambda_1=\lambda_2$ or $\lambda_1+\lambda_2=0$.
The lemma is proved. 
\end{proof}
\bigskip
\begin{tikzpicture}
	\begin{axis}[
		width=16cm,          
		height=5cm, 
		xlabel = $x$,
		ylabel = $y$,
		domain = -3*pi:3*pi,      
		samples = 100,            
		ymin = -1.5,              
		ymax = 1.5,               
		axis lines = middle,      
		grid = both,              
		xtick = {-6.2832, -3.1416, 0, 3.1416, 6.2832}, 
		xticklabels = {$-2\pi$, $-\pi$, $0$, $\pi$, $2\pi$}, 
		legend pos = north east    
		]
		\addplot[black, thick] {sin(deg(x))};
		
		\addplot[red, dashed, thick] {-0.2176*x};
		\addplot[red, dashed, thick] {0.1274*x};

	\end{axis}
	\node[anchor=north] at (current axis.south) {\large\bfseries Figue 1};
\end{tikzpicture}

By Lemma \ref{thm4.1} (or Theorem 4.2 of \cite{LLP21}), for given a spectrum $\Lambda$ of $\rho_{t_1,t_2}$, we have $\lambda_1\neq\lambda_1'$ and $\lambda_2\neq \lambda_2'$ for any $(\lambda_1,\lambda_2)\neq(\lambda_1',\lambda_2')\in\Lambda$. 
This implies that $\tau_j((\lambda_1,\lambda_2))=\lambda_j$ is an injective map on $\Lambda$. 
Then we obtain the following lemma.

\begin{lem}\label{thm4.2}
	Let $\rho_{t_1,t_2}$ be a spectral measure with a spectrum $\Lambda$. 
	Then $\tau_j(\Lambda)$ is a tight frame-spectrum of $\mathcal{L}_{[0,1]}$ with frame bounded $2$, where $$\tau_j(\Lambda):=\{\lambda_j:(\lambda_1,\lambda_2)\in\Lambda\}  
	$$ and $j=1,2$. 
\end{lem}
\begin{proof}
Let $f(x,0)=g(x-t_1)$ for $x\in[t_1,t_1+1]$, and $f(0,y)=0$ for all $y\in[t_2,t_2+1]$ with $g(x)\in L^2([0,1])$.  
It follows from Parseval's theorem and  $f(x,y)\in L^2(\rho_{t_1,t_2})$ that 
\begin{align*}
\int |f(x,y)|^2d\rho_{t_1,t_2}=\frac12\int_{t_1}^{t_1+1}|f(x,0)|^2dx
&=\sum_{(\lambda_1,\lambda_2)\in\Lambda}
\Big|\frac12\int_{t_1}^{t_1+1}f(x,0)e^{-2\pi i\lambda_1x}dx\Big|^2
\\
&=\frac14\sum_{\lambda_1\in\tau_1(\Lambda)}\Big|\int_0^{1}g(x)e^{-2\pi i\lambda_1x}dx\Big|^2.
\end{align*}
This yields that $$
2\int_0^1|g(x)|^2dx=\sum_{\lambda_1\in\tau_1(\Lambda)}\Big|\int_0^{1}g(x)e^{-2\pi i\lambda_1x}dx\Big|^2
$$
for $g(x)\in L^2([0,1]                                                                                                                                                                                                                                                                                                                                                                                                                                                                                                                                                                                                                                                                                                                                                                                                                                                                            )$. 
Therefore, $\tau_1(\Lambda)$ is a tight frame-spectrum of $\mathcal{L}_{[0,1]}$. By similar arguments, we conclude that  $\tau_2(\Lambda)$ is tight frame-spectrum of $\mathcal{L}_{[0,1]}$, which completes the proof.
\end{proof}
\begin{prop}\label{thm4.3}
	The Plus space $L^2(\rho_{-\frac12})$ admits no exponential orthogonal basis. 
\end{prop}
\begin{proof}
	We prove the proposition by a contradiction. 
Let $\Lambda$ be a spectrum of $\rho_{-\frac12}$. 
From Lemma \ref{thm4.2}, we obtain that $\tau_1(\Lambda)$ is a tight frame-spectrum of $\mathcal{L}_{[0,1]}$ with frame bound $2$, where $\tau_1(\Lambda)=\{\lambda_1:(\lambda_1,\lambda_2)\in\Lambda\}$. 
It follows from Theorem \ref{thm1.1} or Remark \ref{thm1.2} that 
$g_{\min}(\tau_1(\Lambda))\leq \frac{1}{\sqrt{2}}<0.8.$
Then we can find some $(\lambda_1,\lambda_2),(\lambda_1',\lambda_2')\in\Lambda$ such that $0<|\lambda_1'-\lambda_1|<0.8$. 
So, by Lemma \ref{thm4.1} and $(\lambda_1',\lambda_2')-(\lambda_1,\lambda_2)\in\mathcal{Z}(\widehat{\rho_{-\frac12}})$, $\lambda_1'-\lambda_1=\lambda_2'-\lambda_2$ or $\lambda_1'-\lambda_1=\lambda_2-\lambda_2'$. 
Finally, \eqref{(4.1)} tells us that 
$$
0=\frac{\sin\pi(\lambda_1'-\lambda_1)}{\lambda_1'-\lambda_1}+\frac{\sin\pi(\lambda_2'-\lambda_2)}{\lambda_2'-\lambda_2}=\frac{2\sin\pi(\lambda_1'-\lambda_1)}{\lambda_1'-\lambda_1},
$$
which yields that $\lambda_1'-\lambda_1=0$, a contradiction. 
\end{proof}

Next, we will further investigate the case \( (t_1,t_2)\neq(-\frac12,-\frac12) \). Adopting the strategy from \cite{KW25}, we will establish periodicity by utilizing the finite complexity property of the tiling structure. 
A discrete set $\Lambda\subset\Bbb R$ is said to have \textit{finite local complexity} if $g_k(\Lambda)$ take only finitely many different values. 
\begin{lem}\label{thm4.4}
		Let $\rho_{t_1,t_2}$ be a spectral measure with a spectrum $\Lambda$. 
		Suppose that $(t_1,t_2)\neq(-\frac12,-\frac12)$.
	Then there exists $K\in\Bbb Z\setminus\{0\}$ such that  $$\tau_1(\Lambda)=\tau_1(\Lambda)+K, 
	$$ 
	where $\tau_1(\Lambda)$ be given as Lemma  \ref{thm4.2}.
\end{lem}
\begin{proof}
	With loss of generality, we assume $t_2\neq-\frac12$. 
	From Lemma \ref{thm4.2}, we know that $\tau_1(\Lambda)$ is a tight frame-spectrum of $\mathcal{L}_{[0,1]}$. 
	Write $$
	\Lambda=\{(\lambda_k,\lambda_k'):k\in\Bbb Z\}
	$$
	with $\cdots<\lambda_{-1}<\lambda_0<\lambda_1<\cdots$. 
	Recall that $g_k(\tau_1(\Lambda))=\lambda_k-\lambda_{k-1}$ and $g_k(\tau_2(\Lambda))=\lambda_k'-\lambda_{k-1}'$. 
	Due to Lemma \ref{thm4.2} and Corollary \ref{thm3.2}, one has 
	$g_{k}(\tau_j(\Lambda))\leq \pi^2+2$ for $j=1,2$. 
	If $g_{k}(\tau_j(\Lambda))<0.8$, then, by Lemma \ref{thm4.1}, $$
	g_{k}(\tau_1(\Lambda))\pm g_{k}(\tau_2(\Lambda))=0. 
	$$
	As \eqref{(4.1)} and  $$(g_{k}(\tau_1(\Lambda)),g_{k}(\tau_2(\Lambda)))\in\mathcal{Z}(\widehat{\rho_{t_1,t_2}}), $$
	one has $T(g_{k}(\tau_1(\Lambda)),g_{k}(\tau_2(\Lambda)))\in\Bbb Z$, which implies that 
	\begin{align*}
		|g_{k}(\tau_j(\Lambda))| \geq \min\left\{\frac{1}{2|t_1-t_2|}, \frac{1}{2|t_1+t_2+1|}\right\},\;\;\forall\;\;j=1,2.
	\end{align*}
	
	So we can find a constant number $C>0$ such that $C^{-1}\leq g_k(\tau_j(\Lambda))\leq C$ for all $k\in\Bbb Z$ and $j=1,2$. 
	Then $T(g_k(\tau_1(\Lambda)),g_k(\tau_2(\Lambda)))$ is also bounded, and the choice of $T(g_k(\tau_1(\Lambda)),g_k(\tau_2(\Lambda)))$ is finitely many, as $T(g_k(\tau_1(\Lambda)),g_k(\tau_2(\Lambda)))$ is an integer.  Write $T(g_k(\tau_1(\Lambda)),g_k(\tau_2(\Lambda)))=s\neq0$, the $g_k(\tau_1(\Lambda))$ must satisfy 
	the equation $$
	\frac{\sin\pi x}{\pi x}\pm \frac{\sin\pi(\frac{2t_1+1}{2t_2+1}x-\frac{s}{2t_2+1})}{\pi(\frac{2t_1+1}{2t_2+1}x-\frac{s}{2t'+1})}=0.
	$$
	Consequently, the non-constant analytic function on the right-hand side vanishes at points $g_k(\tau_1(\Lambda)) \in [C^{-1}, C]$, which implies that $g_k(\tau_1(\Lambda))$ can only take finitely many possible values. 
	Hence $\tau_1(\Lambda)$ is a set of finite local complexity. 
	Then, similar to the arguments of Lemma 5.4 of \cite{KW25}, $\tau_1(\Lambda_1)$ is a periodic set with a period in $\Bbb Z$. 
\end{proof}

\begin{prop}\label{thm4.5}
		Let $\rho_{t_1,t_2}$ be a spectral measure with a spectrum $\Lambda$. 
 Then $t_1-t_2\in\Bbb Z\setminus\{0\}$ or $t_1+t_2\in\Bbb Z\setminus\{-1\}$. 
\end{prop}
\begin{proof}
	From Proposition \ref{thm4.3}, we assume that $(t_1,t_2)\neq(-\frac12,-\frac12)$.
It follows from Lemma \ref{thm4.2} and Theorem \ref{thm1.1} that 
$g_{\min}(\tau_1(\Lambda))\leq \frac{\sqrt{2}}{2}$. 
By Lemma \ref{thm4.1} and the definition of $g_{\min}(\tau_1(\Lambda))$, with loss of generality, we assume $\mathbf{0},(\lambda_1,\lambda_2)\in\Lambda$ with $\lambda_1\in(0,0.8)$,  $\lambda_1\pm\lambda_2=0$ and 
\begin{align}\label{(4.3)}
	T(\lambda_1,\lambda_2)\in\Bbb Z\setminus2\Bbb Z.
\end{align} 
Lemma \ref{thm4.4} and $\mathbf{0}\in\Lambda$ allow us to assume that there exists $(n,m)\in\Lambda\setminus\{\mathbf{0}\}$. 
$$
	(-1)^{T(n-\lambda_1,m-\lambda_2)}\frac{\sin\pi(n-\lambda_1)}{n-\lambda_1}+	\frac{\sin\pi(m-\lambda_2)}{m-\lambda_2}=0.
$$
For each $(n,m)\in\Lambda\cap\Bbb Z^2$, by $(n-\lambda_1,m-\lambda_2)\in\mathcal{Z}(\widehat{\rho})$ and $|\lambda_1|=|\lambda_2|$, then 
\begin{align}\label{(4.4)}
|n-\lambda_1|=|m-\lambda_2|.
\end{align}
For any $(\lambda_1',\lambda_2')\in\Lambda\setminus\Bbb Z^2$, by $(\lambda_1',\lambda_2')-(n,m)\in\mathcal{Z}(\widehat{\rho_{t_1,t_2}})$, then 
\begin{align}\label{(4.5)}
\frac{\sin(\lambda_1'-n)\pi}{(\lambda_1'-n)\pi}\pm \frac{\sin(\lambda_2'-m)\pi}{(\lambda_2'-m)\pi}=0.
\end{align}
We prove the proposition by dividing two cases: $\lambda_1=\lambda_2$ and $\lambda_1=-\lambda_2$.

Suppose that $\lambda_1=\lambda_2$. 
If $\lambda_1=\frac12$, by \eqref{(4.3)}, one has 
$$
T(\lambda_1,\lambda_2)=t_1-t_2\in\Bbb Z\setminus2\Bbb Z.
$$
Now, we assume $\lambda_1\neq\frac12$, then \eqref{(4.4)} yields that $n=m$. 
Noting the choice of $(n,n)$ are infinitely many, denote as $\{(n_k,n_k)\}\subset\Lambda$ such that the symbols in \eqref{(4.5)} maintain uniform interpretation, then \eqref{(4.5)} becomes  
$$
\frac{\lambda_2'-n_k}{\lambda_1'-n_k}\sin\lambda_1'\pi\pm\sin\lambda_2'\pi=\lim_{k\rightarrow\infty}\frac{\lambda_2'-n_k}{\lambda_1'-n_k}\sin\lambda_1'\pi\pm\sin\lambda_2'\pi=\sin\lambda_1'\pi\pm\sin\lambda_2'\pi=0,
$$
which implies that $\lambda_1'=\lambda_2'$. 
This proves $\Lambda\subset\{(x,x):x\in\Bbb R\}$. 
This means $T(\lambda_1',\lambda_2')\in2\Bbb Z+1$ when $(\lambda_1',\lambda_2')\in\Lambda-\Lambda$ with $\lambda_1'\not\in\Bbb Z$. 
Then $\tau_1(\Lambda)\subset\Bbb Z\cup(\lambda_1+\Bbb Z)$. 
If $1\not\in\tau_1(\Lambda)$, by Lemma \ref{thm4.2}, then 
\begin{align*}
	2=	2\int_0^1|e^{-2\pi i x}|^2dx&=\sum_{\lambda\in(\Bbb Z\cup(\Bbb Z+\lambda_1))\cap\tau_1(\Lambda)}|\widehat{\mathcal{L}}_{[0,1]}(\lambda+1)|^2
	\\
	&=\sum_{\lambda\in(\Bbb Z+\lambda_1)\cap\tau_1(\Lambda)}|\widehat{\mathcal{L}}_{[0,1]}(\lambda+1)|^2\leq 1,
\end{align*}
a contradiction, as $\Bbb Z+\lambda_1$ is a spectrum of $\mathcal{L}_{[0,1]}$. 
So $1\in\tau_1(\Lambda)$, then 
$$
2(t_1-t_2)=T(1,1)=T(\lambda_1,\lambda_2)+T(1-\lambda_1,1-\lambda_2)\in2\Bbb Z,
$$
i.e., $t_1-t_2\in\Bbb Z$. 
By \eqref{(4.3)}, we know that $t_1\neq t_2$, which proves the case. 

Now, we consider the case $\lambda_1=-\lambda_2$. 
If $\lambda_1=\frac12$, by \eqref{(4.3)}, one has 
$$
T(\lambda_1,\lambda_2)=t_1+t_2+1
\in\Bbb Z\setminus2\Bbb Z.
$$
Now, we assume $\lambda_1\neq\frac12$, then \eqref{(4.4)} yields that $n=-m$. 
Similar to the arguments of case $\lambda_1=\lambda_2$, we can get $\Lambda\subset\{(x,-x):x\in\Bbb R\}$,  $\tau_1(\Lambda)\subset\Bbb Z\cup(\lambda_1+\Bbb Z)$ and  $1\in\tau_1(\Lambda)$. 
Hence
$$
2(t_1+t_2+1)=T(1,1)=T(\lambda_1,\lambda_2)+T(1-\lambda_1,1-\lambda_2)\in2\Bbb Z,
$$
i.e., $t_1+t_2+1\in\Bbb Z$. 
Finally, using \eqref{(4.3)} again, we reduce  that $t_1+t_2\neq-1$, which proves the proposition. 
\end{proof}
The following theorem is a fundamental criterion. 
\begin{thm}[\cite{JP98}]\label{thm4.6}
	Let $\mu$ be a Borel probability measure on $\Bbb R^n$. 
	Then a countable set $\Lambda\subset \Bbb R^n$ is a spectrum of $L^2 (\mu)$ if and only if $\sum_{\lambda\in\Lambda}|\widehat{\mu}(\xi+\lambda)|^2=1$ for all $\xi\in\mathbb{R}^n$.
\end{thm}
Before the proof of Theorem \ref{thm1.7}, we need the following equality. 
\begin{lem}\label{thm4.7}
	Let $\xi_1,\xi_2\in\Bbb R$. 
	Then
	$$
	\sum_{n\in\Bbb Z}\frac{\sin\pi(\xi_1+n)}{\pi(\xi_1+n)}\frac{\sin\pi(\xi_2+n)}{\pi(\xi_2+n)}=\frac{\sin(\xi_2-\xi_1)\pi}{(\xi_2-\xi_1)\pi},  
	$$
where $\frac{\sin(\xi_2-\xi_1)\pi}{(\xi_2-\xi_1)\pi}=1$ when $\xi_1=\xi_2$. 
\end{lem}
\begin{proof}
	The lemma follows directly from the Poisson summation formula to the function 
	$$
	\frac{\sin\pi(\xi_1+x)}{\pi(\xi_1+x)}\frac{\sin\pi(\xi_2+x)}{\pi(\xi_2+x)}. 
	$$
\end{proof}

At the end of this section, we prove Theorem \ref{thm1.7} by showing the following proposition. 

\begin{prop}
	Let $\rho_{t_1,t_2}$ be given as \eqref{(1.2)}, and let $$
	\Lambda:=\begin{cases}
\{(n,n),(n+\frac{1}{2(t_1-t_2)},n+\frac{1}{2(t_1-t_2)}):n\in\Bbb Z\}, & t_1-t_2\in\Bbb Z\setminus\{0\};
\\
\{(n,-n),(n+\frac{1}{2(t_1+t_2+1)},-n-\frac{1}{2(t_1+t_2+1)}):n\in\Bbb Z\}, & t_1+t_2\in\Bbb Z\setminus\{-1\}.
	\end{cases}
	$$
	If $t_1-t_2\in\Bbb Z\setminus\{0\}$ or $t_1+t_2\in\Bbb Z\setminus\{-1\}$, then $\Lambda$ is a spectrum of $\rho_{t_1,t_2}$. 
\end{prop}
\begin{proof}
	It follows that 
	\begin{align*}
\sum_{\lambda\in\Lambda}|\widehat{\rho_{t_1,t_2}}(\xi+\lambda)|^2&=	\frac14\sum_{(\lambda_1,\lambda_2)\in\Lambda}\Big|
		\int_{t_1}^{t_1+1}e^{-2\pi i(\xi_1+\lambda_1)x}dx+		\int_{t_2}^{t_2+1}e^{-2\pi i(\xi_2+\lambda_2)x}dx\Big|^2
		\\
		&=\frac14\sum_{(\lambda_1,\lambda_2)\in\Lambda}(|\widehat{\mathcal{L}}_{[t_1,t_1+1]}(\xi_1+\lambda_1)|^2+|\widehat{\mathcal{L}}_{[t_2,t_2+1]}(\xi_2+\lambda_2)|^2+C),
	\end{align*}
where 
	\begin{align*}
		C:=&\frac14\sum_{(\lambda_1,\lambda_2)\in\Lambda}2\cos\pi(T\xi+T(\lambda_1,\lambda_2))\frac{\sin\pi(\xi_1+\lambda_1)}{\pi(\xi_1+\lambda_1)}\frac{\sin\pi(\xi_2+\lambda_2)}{\pi(\xi_2+\lambda_2)}.
	\end{align*}
	Since $T(\lambda)\in2\Bbb Z$ when $\lambda\in\Lambda\cap\Bbb Z^2$, and $T(\lambda)\in\Bbb Z\setminus2\Bbb Z$ when $\lambda\in\Lambda\setminus\Bbb Z^2$, we get that $\frac{2C}{\cos T\xi}$ can be written as  
		\begin{align*}
	\sum_{(\lambda_1,\lambda_2)\in\Lambda\cap\Bbb Z^2}\Big(\frac{\sin\pi(\xi_1+\lambda_1)}{\pi(\xi_1+\lambda_1)}\frac{\sin\pi(\xi_2+\lambda_2)}{\pi(\xi_2+\lambda_2)}-\frac{\sin\pi(\xi_1'+\lambda_1)}{\pi(\xi_1'+\lambda_1)}\frac{\sin\pi(\xi_2'+\lambda_2)}{\pi(\xi_2'+\lambda_2)}\Big), 
	\end{align*}
and $$\xi_i'=\xi_i+\begin{cases}
\frac{1}{2(t_1-t_2)}, & t_1-t_2\in\Bbb Z\setminus\{0\};
\\
(-1)^{i-1
}\frac{1}{2(t_1+t_2+1)}, & t_1+t_2\in\Bbb Z\setminus\{-1\}.
\end{cases}. $$
It follows from Lemma \ref{thm4.7} that 
$$
	\sum_{(\lambda_1,\lambda_2)\in\Lambda\cap\Bbb Z^2}\frac{\sin\pi(\xi_1+\lambda_1)}{\pi(\xi_1+\lambda_1)}\frac{\sin\pi(\xi_2+\lambda_2)}{\pi(\xi_2+\lambda_2)}=\begin{cases}
		\frac{\sin(\xi_2-\xi_1)\pi}{(\xi_2-\xi_1)\pi}, & t_1-t_2\in\Bbb Z\setminus\{0\};
		\\
		\frac{\sin(\xi_2+\xi_1)\pi}{(\xi_2+\xi_1)\pi}, & t_1+t_2\in\Bbb Z\setminus\{-1\}.
	\end{cases}
$$
Observing that substituting $(\xi_1,\xi_2)$ with $(\xi_1',\xi_2')$ in the above equation preserves its validity, we conclude that $C\equiv0$.
Therefore 
	\begin{align*}
	\sum_{\lambda\in\Lambda}|\widehat{\rho_{t_1,t_2}}(\xi+\lambda)|^2=\frac14\sum_{(\lambda_1,\lambda_2)\in\Lambda}(|\widehat{\mathcal{L}}_{[t_1,t_1+1]}(\xi_1+\lambda_1)|^2+|\widehat{\mathcal{L}}_{[t_2,t_2+1]}(\xi_2+\lambda_2)|^2)=1,
\end{align*}
	So, from Theorem \ref{thm4.6}, $\Lambda$ is a spectrum of $\rho_{t_1,t_2}$. 
	This proves the proposition. 
\end{proof}

\end{document}